\documentclass{amsart}

\newtheorem{theorem}{Theorem}[section]

\theoremstyle{definition}

\newtheorem{example}[theorem]{Example}

\newtheorem{cor}[theorem]{Corollary}

\theoremstyle{remark}
\newtheorem{remark}[theorem]{Remark}

\numberwithin{equation}{section}

\begin{document}

\title{Some calibrated surfaces in manifolds with density}

\author{Doan The Hieu}
\address{Hue geometry Group}
\curraddr{ College of Education, Hue University
 34 Le Loi, Hue, Vietnam}
\email{ dthehieu@yahoo.com}
\thanks{The author was supported in part by a Nafosted Grant}


\date{March 20, 2010 and, in revised form .}



\begin{abstract}
Hyperplanes, hyperspheres and hypercylinders in $\Bbb R^n$ with suitable densities are proved to be weighted area-minimizing by a calibration argument.
\end{abstract}

\maketitle



\section{Introduction}
A manifold with density is a Riemannian manifold $M$  endowed with a positive function (density) $e^{\psi}$  used to weight both volume and perimeter. The weighted volume and perimeter elements are defined as $e^\psi dV$ and $e^\psi dA,$ where $dV$ and $dA$ are the Riemannian volume and perimeter elements.

  A typical example of  such manifolds is Gauss space $G^n,$ that is $\Bbb R^n$ with Gaussian probability density $(2\pi)^{-\frac n2}e^{-\frac{r^2}2}.$ Gauss space has many applications to probability and statistics. For more details about manifolds with density,  we refer the reader to \cite{mo1}, \cite{mo2}, \cite {RCBM} and the entry ``Manifolds  with density'' at Morgan's blog http://blogs.williams.edu/Morgan/.

  Manifolds with density are a good setting to extend some variational problems in geometry such as isoperimetric problems, minimizing networks, minimizing surfaces\ldots. It is also good to consider some problems concerned with notions of curvature.

   Following Gromov (\cite[p. 213]{gr}),  the natural generalization of the mean curvature, called weighted mean curvature, of a hypersurface in a manifold with density $e^{\psi}$ is defined as
\begin{equation}      H_{\psi}=H-\frac 1{n-1}\frac{d\psi}{d{\bf n}},\end{equation}
where $H$ is the classical mean curvature and {\bf n} is the normal vector field of the hypersurface.
The definition of the weighted mean curvature is fit for the first variation of weighted perimeter of a smooth region (see \cite{co2}, \cite{RCBM}).


For a stationary (i.e. with vanishing first variation of perimeter) smooth open set $\Omega\subset\Bbb R^{n}$ endowed with a smooth density $e^\psi,$ let $N$ be the inward unit normal vector to
$\Sigma= \partial\Omega,$ and $H_{\psi}$ be the
constant weighted mean curvature of $\Sigma$ with respect to $N.$ Consider a variation of $\Omega$ with associated
vector field $X = uN$ on $\Sigma.$

 Bayle \cite{ba} computed the second variation formula of the functional $P-H_{\psi}V$ for any variation of a stationary set  and obtained the following formula
\begin{equation}\label{svf}(P-H_{\psi}V)''=Q_{\psi}(u,u)=\int_{\Sigma}e^\psi
(|\nabla_{\Sigma}u|^2-|
\sigma|^2u^2)da+\int_{\Sigma}e^\psi u^2(\nabla^2\psi)(N,N)da,\end{equation}
where $\nabla_{\Sigma}u$ is the gradient of $u$ relative to $\Sigma$,  $|\sigma|^2$ is the squared sum of the principal curvatures of $\Sigma$ and  $\nabla^2\psi$ is the Euclidean Hessian of $\psi.$

The situation is the same as in the Euclidean case, $\Omega$ is stable ($P''(0)\ge 0$) if and only if $Q_\psi(u,u)\ge 0$ for a variation satisfying the condition $\int e^\psi udA=0$. This condition means  that $u$ is orthogonal to $e^\psi$ in $L^2(\Sigma)$ and it is proved that any such $u$ is the normal component of a vector field associated to a volume-preserving variation of $\Omega$ (see \cite{baca}, \cite{RCBM}).

In $\Bbb R^n$  with a log-convex spherical density, balls about the origin are stable and it is conjectured that they are the only isoperimetric regions (see \cite{RCBM}).

We are interested in the question of what conditions on density make some constant weighted mean curvature hypersurfaces stable and weighted area-minimizing. Weighted area-minimizing means having least weighted perimeter in a homology class (Section \ref{sec2}) or under compact, weighted-volume-preserving deformations (Sections \ref{sec3}, \ref{sec4}, \ref{sec5}). We consider three cases: hyperplanes in $\Bbb R^n$ with a smooth density $\delta=e^{\varphi(x)+ \psi(x_n)},$ where $x=(x_1, x_2,\ldots, x_{n-1});$  hyperspheres in $\Bbb R^n-\{O\}$ with a smooth spherical density and hypercylinders in $\Bbb R^n-\{O\}\times\Bbb R^k$ with a smooth cylindrical density. The proofs are an application of Stokes' theorem as in the calibration method.

We begin, in Section \ref{sec2},  with The Fundamental Theorem of Weighted Calibrations and some applications including a proof that a weighted minimal hypergraph in $\Bbb R^n$ with a non-depending on the last coordinate density is weighted area-minimizing in its homology class. Some other examples of weighted calibrated submanifolds are also presented in this section.


\section{Calibrations on manifolds with density} \label{sec2}

Let $M$ be a Riemannian manifold with a smooth density $e^\psi$ and $\Phi$ is a $k$-differential form. We define the weighted exterior derivative with density $d_\psi$ as follows
$$d_\psi(\Phi):=e^{-\psi}de^\psi\Phi.$$
The definition of $d_\psi$ appeared first in \cite{li}  and \cite{wi}. A $k$-differential form $\Phi$ is called $d_\psi$-closed if $d_\psi(\Phi)=0$ and this is equivalent to $de^\psi\Phi=0.$ A $d_\psi$-closed differential form is called a weighted calibration if it has comass one. For the definition of the comass of a $k$-differential form, and calibrated geometry we refer to \cite{hala}. A $k$-submanifold $N$ of $M$ is called a weighted calibrated submanifold, calibrated by the weighted calibration $\Phi,$ if $\Phi$ attains its maximum on tangent planes of $N$ almost everywhere. Here the Riemannian volume  and the weighted volume (denoted by $\text{Vol}_\psi$) of $N$ are $\int_N\Phi$ and $\int_Ne^\psi\Phi,$ respectively. By a similar proof as that of The Fundamental Theorem of Calibrations  with density 1 (see \cite{hala}, \cite [Section 6.4, 6.5] {mo2}), we have
\begin{theorem}\label{theo21}
Every weighted calibrated submanifold with or without boundary is weighted area-minimizing in its homology class.
\end{theorem}

\begin{proof}
Let $N$ and $\overline N$ be $k$-submanifolds in the same homology class, i.e. $\partial N=\partial\overline N$ and $N-\overline N=\partial A$ for some $(k+1)$-chain $A.$ Suppose that $N$ is calibrated by weighted calibration $\Phi.$ Then
\begin{equation}\text{Vol}_\psi(N)-\text{Vol}_\psi\overline N\le \int_Ne^\psi\Phi-\int_{\overline N}e^\psi\Phi=\int_{N-\overline N}e^\psi\Phi=\int_{\partial A}e^\psi\Phi.
\end{equation}
Because $\Phi$ is $d_\psi$-closed, by Stokes' theorem the last term vanishes and the theorem is proved.

\end{proof}

The following examples illustrate some applications of Theorem \ref{theo21}.

\begin{example}\label{ex22}
It is well known that in $\Bbb R^n$ with a constant density, minimal hypersurfaces are area-minimizing locally. We will show that the result is also true in some cases of non-constant density.


 Suppose $S$ be the minimal hypergraph defined by $x_n=f(x_1, x_2, \ldots, x_{n-1})$  in $\Bbb R^n=\Bbb R^{n-1}\times \Bbb R$ over the domain $U\subset\Bbb R^{n-1},$ where $\Bbb R^{n-1}$ and $\Bbb R$ endowed  density $e^\psi$ and $1$, respectively. Let ${\bf n}$ is its unit normal field and consider the smooth extension  of ${\bf n}$ by the translation along $x_n$-axis, also denoted by ${\bf n},$ in the cylinder $U\times \Bbb R.$
  
  It is not difficult to see that the $(n-1)$-differential form defined by
 $$w(X_1, X_2,\ldots, X_{n-1})= \det(X_1, X_2,\ldots, X_{n-1}, {\bf n})$$
 where $X_i,\ i=1,2,\ldots, n-1$ are smooth vector fields on $S,$ has comass 1.

 Moreover,
\begin{align}\nonumber
d(e^\psi w)&= \text{div}(e^\psi{\bf n})dV_M\\
 &= (e^\psi\text{div}({\bf n})+e^\psi\langle\nabla\psi, {\bf n}\rangle)dV_M\\
 &=(-e^\psi(n-1)H+e^\psi\langle\nabla\psi, {\bf n}\rangle)dV_M=0,
\end{align}
because $S$ is minimal.
Thus, $w$ is a weighted calibration. Obviously, $w$ calibrates $S$ in $U\times \Bbb R.$

\end{example}

\begin{example}
Consider  a product $M\times M',$ where $M$ is a Riemannian $n$-manifold with density 1 and $M'$ is another $m$-manifold with density $e^\psi.$ Denote by $dV_{M'}$ the  Riemannian volume element on $M'.$ Let $\Phi$ be a $k$-calibration on $M$ calibrating $k$-submanifold $N\subset M.$ Then $\Phi\wedge dV_M'$ is a weighted $(n+k)$-calibration in $M\times M'$ calibrating $N\times M'.$ Below are some concrete examples
\begin{enumerate}
\item Let $\Phi=1;$ then $dV_{M'}$  calibrates
$\{x\}\times M'$ for any $x\in M.$
\item In $\Bbb R^n$ with density independent of $m$ last coordinates $x_{n-m+1},\ldots, x_n,\ \Phi=dx_1\wedge dx_2\wedge\ldots\wedge dx_{n-m}$ is a weighted calibration calibrating every $(n-m)$-plane $\{x_i=\text{const.}, i=n-m+1,\ldots, n\}$ (see also \cite{cami}, Section 2.1).
\item The 3-covector $\Phi=(e^*_1\wedge e^*_2+e^*_3\wedge e^*_4)\wedge e^*_5$ is a calibration in $\Bbb R^5$ with density 1. With complex structure $Je_1=e_2,\ Je_3=e_4$ on $\Bbb R^4,$\ $\Phi$ calibrates every complex curve in $\Bbb R^4$  times $\Bbb R.$ $\Phi$ is also a weighted calibration on $\Bbb R^5$ with a density depending only on the last coordinate $x_5$ and calibrates the same 3-submanifolds as in the case with density 1.
\end{enumerate}
\end{example}

\begin{example}
Consider the cylindrical coordinate system $(\rho,\varphi,z)$ on $\Bbb R^2-{O}\times\Bbb R$ (see \ref{sec4}) with density $\rho^{-1}.$ The area element $dA=\rho d\varphi\wedge dz,$ is a weighted calibration. It calibrates  every cylinder about the $z$-axis.
\end{example}

\begin{example}
Consider the spherical coordinate system $(r,\varphi)$  on $\Bbb R^n-\{O\}$ (see Section \ref{sec3}) with density $r^{1-n}.$ The perimeter element $dA=r^{n-1}d\omega,$  is a weighted calibration calibrating every  hypersphere about the origin.

\end{example}

\section{Weighted minimizing hyperspheres}\label{sec3}

Consider $\Bbb R^n$ with a spherical density $e^{\psi(r)}.$ If the density is log-convex, hyperspheres about the origin are stable and it is conjectured that they are the only isoperimetric regions. This conjecture was proved in the real line, in $\Bbb R^n$ with specific density $e^{r^2}$ and in $\Bbb R^2$ with density $e^{r^p},\ p\ge 2.$ The conjecture is still open in general (see \cite{bo}, \cite{RCBM}, \cite{mamo} and the entry ``The Log-Convex Density Conjecture'' at Morgan's blog http://blogs.williams.edu/Morgan/).

 We consider in $\Bbb R^n-\{O\}$ the spherical coordinates $(r,\varphi),$ where $\varphi=(\varphi_1,\ldots,\varphi_{n-1}) $ and
\begin{align}\nonumber
r&=|x|,\\
x_1&=r\cos \varphi_1,\nonumber\\
x_{k}&=r\sin\varphi_1\sin\varphi_2\ldots\sin\varphi_{k-1}\cos \varphi_{k},\ \ \ \text{for} \ \ k=2,\ldots, n-1,\\ \nonumber
x_{n}&=r\sin\varphi_1\sin\varphi_2\ldots\sin\varphi_{n-2}\sin \varphi_{n-1.}\nonumber
\end{align}
Let $dV=r^{n-1}d\Omega$ be the volume element  and $dA=r^{n-1}d\omega$ be the perimeter element, where $d\Omega$ and $d\omega$ are the volume elemnent for the unit ball and the perimeter element for the unit hypersphere, respectively.

Taking the exterior derivative of the differential form $\Phi=e^{\psi(r)}dA,$ we get
$$d\Phi= (\psi'+\frac{n-1}{r})e^\psi dV.$$

Denote by $B(r)$ and $S(r)$ the ball and hypersphere about the origin with radius $r.$ We have

\begin{theorem}\label{theo31}
 In $B(r_1)-B(r_0),\ r_1>r_0, $ with spherical density $e^\psi,$ if $r^{n-1}e^{\psi(r)}$ is log-convex,
 every hypersphere about the origin is weighted area-minimizing.
\end{theorem}

\begin{proof}
Because $ \left[\log(r^{n-1}e^{\psi(r)})\right]''=(\psi'+\frac{n-1}{r})'=\psi''-\frac{n-1}{r^2}\ge 0,$ we see that  $(\psi'(r)+\frac{n-1}{r})$ is increasing in $B(r_1)-B(r_0).$

Condider a hypersphere $S(r)$ in $B(r_1)-B(r_0)$ and let $\overline S$ be a competitor of $S(r)$ under a compact, weighted-volume-preserving deformation.  Denote by $R^+$ and $R^-$ the regions bounded by $S(r)$ and $\overline S$ lying outside and inside the ball $B(r),$ respectively, and set $R=R^+\cup R^-.$

Because the enclosed weighted volume is preserved

$$\int_{R^+}e^\psi dV_{\Bbb R^{n}}=\int_{R^-}e^\psi dV_{\Bbb R^{n}}.$$

Thus, we have
\begin{align}\text{Area}_\psi(\overline S)-\text{Area}_\psi(S(r))&\ge \int_{\overline S}\Phi-\int_{S(r)}\Phi=\int_{\overline S-S(r)}\Phi=\int_Rd\Phi\nonumber\\
&=\int_{R^+}d\Phi-\int_{R^-}d\Phi\\
&> \left(\psi'(r)+\frac{n-1}{r}\right)\left(\int_{R^+}e^\psi dV_{\Bbb R^{n}}-\int_{R^-}e^\psi dV_{\Bbb R^{n}}\right)= 0.\nonumber
\end{align}

The theorem is proved
\end{proof}
\begin{cor}
In $\Bbb R^n-\{O\}$ with log-convex spherical density
$e^{\psi(r)},$ if $\psi''(r_0)-\frac{n-1}{r_0^2}> 0,$ then  the
hypersphere  $S(r_0)$ is weighted area-minimizing.
 \end{cor}
 \begin{proof}
Since $\psi''(r_0)-\frac{n-1}{r_0^2}> 0,$ there exists $\epsilon>0,$  such that $\psi''(r)-\frac{n-1}{r^2}> 0$ (or equivalently, $\psi'+\frac{n-1}{r}$ is strict increasing) in $(r_0-\epsilon, r_0+\epsilon).$
\end{proof}
\begin{cor}
In $\Bbb R^n-\{O\}$ with a strongly log-convex spherical density, there exists $r_0>0,$
 such that every hypersphere about the origin in $\Bbb R^n-B(r_0)$
 is weighted area-minimizing.
\end{cor}
\begin{proof}

Since the density is strongly log-convex,  there exists $r_0>0$ such that if $r>r_0,\ \psi''(r)>M>\frac{n-1}{r^2}.$
Thus, for $r>r_0,\ \psi''(r)-\frac{n-1}{r^2}>0.$ By Theorem \ref{theo31} we have the proof.
\end{proof}
In the case of $r_0=0$ and $r_1=\infty,$ we get
\begin{cor} In $\Bbb R^n-\{O\},$ with spherical density $e^\psi(r),$ if $r^{n-1}e^{\psi(r)}$ is log-convex, then  every hypersphere about the origin is weighted area-minimizing.
\end{cor}

\section{Weighted minimizing $k$-hypercylinders}\label{sec4}
Consider the product $\Bbb R^n-\{O\}\times \Bbb R^k,$ where $\Bbb R^n-\{O\}$ endowed with a smooth spherical density $e^{\psi(r)}$ and $\Bbb R^k$ has density 1. We
 call $C(r)=S(r)\times \Bbb R^k$ a $k$-hypercylinder. Because $\nabla\psi$
  is parallel to $\Bbb R^n,$ a $k$-hypercylinder has constant mean curvature.


The second variation formula (\ref{svf}) for $k$-hypercylinders is

\begin{equation}\label{svf1}Q_{\psi}(u,u)=e^\psi\int_{C}|\nabla_{C}u|^2da+
e^\psi(\psi''-\frac{n-1}{r^2})\int_{C}u^2da.\end{equation}

Let $dV$ be the weighted volume element  and $dA$ be the weighted perimeter element in $\Bbb R^n$ as in Section \ref{sec3}. The weighted volume and weighted perimeter elements in $\Bbb R^n-\{O\}\times \Bbb R^k,$
are $dV\wedge dV_{\Bbb R^k}$ and $dA\wedge dV_{\Bbb R^k},$ respectively.

We have $d(dA\wedge dV_{\Bbb R^k})=(\psi'+\frac{n-1}{r})dV\wedge dV_{\Bbb R^k},$ and by a proof as that in Section \ref{sec2}, we get

\begin{theorem}\label{theo41}
 In $C(r_1)-C(r_0),\ r_1>r_0, $ with cylindrical density $e^{\psi(r)},$
 if $r^{n-1}e^{\psi(r)}$ is log-convex, then every hypercylinder is
 weighted area-minimizing.
\end{theorem}

\begin{cor} In $\Bbb R^n-\{O\}\times \Bbb R^k$ with cylindrical denstity $e^{\psi(r)},$
\begin{enumerate}
\item if $\psi''(r_0)-\frac{n-1}{r_0^2}> 0,$ then  the
$k$-hypercylinder  $C(r_0)$ is weighted area-minimizing;
 \item if $\psi$ is strongly convex, there exists $r_0>0$ such that in
$\Bbb R^n-B(r_0)\times\Bbb R^k,$ every $k$-hypercylinder is weighted area-minimizing;
\item if $r^{n-1}e^{\psi(r)}$ is log-convex,
 then every $k$-hypercylinder is weighted area-minimizing.
\end{enumerate}
 \end{cor}


\section{Weighted minimizing hyperplanes}\label{sec5}
   Let $x=(x_1,x_2,\ldots, x_{n-1})$ and consider $\Bbb R^{n}$ endowed with  smooth density $\delta=e^{\varphi(x)+ \psi(x_n)},$ which can be viewed as product space $\Bbb R^{n-1}\times\Bbb R,$ where $\Bbb R^{n-1}$ has smooth density $e^{\varphi(x)}$ and $\Bbb R$ has smooth density $e^{\psi(x_n)}.$ Let $\Sigma$ be the hyperplane determined by the equation  $x_n=a\in\Bbb R.$ It is easy to see that $\Sigma$ has constant mean curvature, $(\nabla^2\psi)(N,N)=\psi''(a)$ and the second variation formula (\ref{svf}) for $\Sigma$ is

\begin{equation}\label{svf1}Q_{\psi}(u,u)=e^\psi\int_{\Sigma}|\nabla_{\Sigma}u|^2da+
e^\psi\psi''(a)\int_{\Sigma}u^2da.\end{equation}

Since hyperplanes are stable in $\Bbb R^n$ with density 1 and $\int_{\Sigma}|\nabla_{\Sigma}u|^2da$ is nonnegative and vanishes for translation, we have the following

\begin{theorem} \label{lem1}In $\Bbb R^n$ with  smooth density $\delta=e^{\varphi(x)+\psi(x_n)},$ the horizontal hyperplane $x_n=a$ is stable if and only if $\psi''(a)\ge 0.$
\end{theorem}

By the same arguments as that of Section \ref{sec3}, we have

\begin{theorem}\label{theo52}
 In $\Bbb R^{n-1}\times (a,b),$ where $(a,b)$ is an interval in
 $\Bbb R,$ with smooth density $\delta=e^{\varphi(x)+ \psi(x_n)},$ if $\psi$ is convex, then horizontal hyperplanes are weighted area-minimizing.

  \end{theorem}

\begin{cor}
In  $\Bbb R^n$ with density $\delta=e^{\varphi(x)+ \psi(x_n)},$ if $\psi''(a)>0,$ then  the horizontal hyperplane $\{x_n= a\}$ is weighted area-minimizing.

 \end{cor}

\begin{cor} \label{cor54}
In  $\Bbb R^n$ with density $e^{r^2}$, where $r$ is the distance from point to the origin, every hyperplane is weighted area-minimizing.
 \end{cor}
\begin{proof}
 By virtue of Theorem \ref{theo52}, hyperplanes  perpendicular to any axis are weighted area-minimizing. Since orthogonal transformations (fixing the origin) preserve $r,$  every hyperplane is weighted area-minimizing.
\end{proof}

\begin{remark}
In  $\Bbb R^n$ with density $e^{cr^2},\ c>0,$  hyperspheres are uniquely isoperimetric (\cite[Theorem 4.1]{bo2}), \cite[Theorem 5.2]{RCBM}). This does not contradict our Corollary \ref{cor54}, in which the behavior at infinity is fixed.

\end{remark}


{\bf Acknowledgements.} We would like to thank  Professor Frank Morgan for bringing the topic to our attention and for his help and discussion.

\bibliographystyle{amsplain}

\end{document}